\documentclass[12pt]{amsart}
\usepackage{amsfonts, amssymb, latexsym, epsfig, epsf, amsmath, amscd}
\usepackage{tikz-cd}
\usepackage{graphicx}
\usepackage{url}

\newtheorem{Theorem}{Theorem}[section]
\newtheorem{Proposition}[Theorem]{Proposition}
\newtheorem{Lemma}[Theorem]{Lemma}

\newtheorem{Corollary}[Theorem]{Corollary}
\newtheorem{Problem}[Theorem]{Problem}
\newtheorem{Definition-Proposition}[Theorem]{Definition-Theorem}
\newtheorem{Main Conjecture}[Theorem]{Main Conjecture}

\theoremstyle{remark}
\newtheorem{Example}[Theorem]{Example}
\newtheorem{prop}{Proposition}[section]
\newtheorem{Remark}[prop]{Remark}
\newcommand{\id}{\mathrm{id}}

\newcommand\Schub{{\mathfrak S}}

\newcommand\Groth{{\mathfrak G}}

\newcommand{\xx}{\mathbf x}
\newcommand{\yy}{\mathbf y}

\newcommand{\caI}{\mathcal{I}}

\newcommand{\caL}{\mathcal{L}}

\newcommand{\caN}{\mathcal{N}}
\newcommand{\caO}{\mathcal{O}}

\newcommand{\caY}{\mathcal{Y}}

\newcommand{\frX}{\mathfrak{X}}

\newcommand{\field}{\mathbb}

\newcommand{\ga}{\alpha}

\newcommand{\C}{{\field C}}

\newcommand{\Z}{{\field Z}}

\theoremstyle{plain}

\newcommand{\excise}[1]{}

\newcommand{\cellsize}{11}
\newlength{\cellsz} 
\newsavebox{\cell}
\sbox{\cell}{\begin{picture}(\cellsize,\cellsize)
\put(0,0){\line(1,0){\cellsize}}
\put(0,0){\line(0,1){\cellsize}}
\put(\cellsize,0){\line(0,1){\cellsize}}
\put(0,\cellsize){\line(1,0){\cellsize}}
\end{picture}}
\newcommand\cellify[1]{\def\thearg{#1}\def\nothing{}%
\ifx\thearg\nothing
\vrule width0pt height\cellsz depth0pt\else
\hbox to 0pt{\usebox{\cell} \hss}\fi%
\vbox to \cellsz{
\vss
\hbox to \cellsz{\hss$#1$\hss}
\vss}}
\newcommand\tableau[1]{\vtop{\let\\\cr
\baselineskip -16000pt \lineskiplimit 16000pt \lineskip 0pt
\ialign{&\cellify{##}\cr#1\crcr}}}
\newcommand{\kellsize}{24}
\newlength{\kellsz} 
\newsavebox{\kell}
\sbox{\kell}{\begin{picture}(\kellsize,\kellsize)
\put(0,0){\line(1,0){\kellsize}}
\put(0,0){\line(0,1){\kellsize}}
\put(\kellsize,0){\line(0,1){\kellsize}}
\put(0,\kellsize){\line(1,0){\kellsize}}
\end{picture}}
\newcommand\kellify[1]{\def\thearg{#1}\def\nothing{}%
\ifx\thearg\nothing
\vrule width0pt height\kellsz depth0pt\else
\hbox to 0pt{\usebox{\kell} \hss}\fi%
\vbox to \kellsz{
\vss
\hbox to \kellsz{\hss$#1$\hss}
\vss}}
\newcommand\ktableau[1]{\vtop{\let\\\cr
\baselineskip -16000pt \lineskiplimit 16000pt \lineskip 0pt
\ialign{&\kellify{##}\cr#1\crcr}}}

\newcommand{\sellsize}{63}
\newlength{\sellsz} 
\newsavebox{\sell}
\sbox{\sell}{\begin{picture}(\sellsize,20)
\put(0,0){\line(1,0){\sellsize}}
\put(0,0){\line(0,1){\sellsize}}
\put(\sellsize,0){\line(0,1){\sellsize}}
\put(0,\sellsize){\line(1,0){\sellsize}}
\end{picture}}
\newcommand\sellify[1]{\def\thearg{#1}\def\nothing{}%
\ifx\thearg\nothing
\vrule width0pt height\sellsz depth0pt\else
\hbox to 0pt{\usebox{\sell} \hss}\fi%
\vbox to \sellsz{
\vss
\hbox to \sellsz{\hss$#1$\hss}
\vss}}
\newcommand\stableau[1]{\vtop{\let\\\cr
\baselineskip -16000pt \lineskiplimit 16000pt \lineskip 0pt
\ialign{&\sellify{##}\cr#1\crcr}}}

\begin{document}
\pagestyle{plain}
\mbox{}
\title{Polynomials for symmetric orbit closures in the flag variety}
\author{Benjamin J.~Wyser}
\author{Alexander Yong}
\address{Dept~of Mathematics, University of Illinois at
Urbana-Champaign, Urbana, IL 61801, USA}
\email{bwyser@uiuc.edu, ayong@uiuc.edu}
\date{September 30, 2014}
\maketitle

\begin{abstract}
In [Wyser-Yong '13] we introduced polynomial representatives of cohomology classes of orbit closures in the flag variety, for the symmetric pair $(GL_{p+q}, GL_p \times GL_q)$. We present analogous results for the remaining symmetric pairs of the form $(GL_n,K)$, i.e., 
$(GL_n,O_n)$ and $(GL_{2n},Sp_{2n})$. We establish the well-definedness of certain representatives from [Wyser '13]. It is also shown that the representatives
have the combinatorial properties of nonnegativity and stability. Moreover, we 
give some extensions to equivariant $K$-theory.
\end{abstract}
\tableofcontents
\section{Introduction}

\subsection{Overview}
This is a sequel to \cite{WyserYong}, where we introduced polynomial representatives of cohomology classes of $K=GL_p\times GL_q$-orbit closures in the flag variety $GL_{p+q}/B$. We 
present analogous definitions and results for
the remaining symmetric pairs $(GL_n,K)$, i.e.,
$(GL_{n},O_{n})$ and $(GL_{2n},Sp_{2n})$. 
Among our results, we establish the well-definedness, nonnegativity and ``stability'' of certain representatives from \cite{Wyser-13b}.

A subgroup $K$ of a reductive algebraic group $G$
is {\bf spherical} if its action on $G/B$ by left translations has finitely many orbits. Such a group is furthermore {\bf symmetric} if $K=G^{\theta}$ is the fixed point subgroup for a holomorphic involution $\theta$ of $G$. The symmetric 
pairs $(G,K)$ are classified in general. The geometry of symmetric orbit closures arises in the representation theory
of real forms $G_{\mathbb R}$ of complex semisimple (reductive) Lie groups $G$.
We refer the reader to \cite[Section~1]{WyserYong} and specifically the 
references therein for more background.

For a torus $S$, an $S$-stable subvariety $X\subset GL_n/B$ admits a class in both $H^{\star}(GL_n/B)$ and $S$-equivariant cohomology $H_S^{\star}(GL_n/B)$.  (For us, $S$ will be a maximal torus of $K$, and $X$ will be a $K$-orbit closure.)
Focusing on the ordinary case for simplicity, we have A.~Borel's isomorphism \cite{Borel}, i.e.,
\[H^{\star}(GL_n/B)\cong {\mathbb Z}[x_1,\ldots,x_n]/I^{S_n},\] 
where $I^{S_n}$
is the ideal generated by elementary symmetric polynomials of positive degree.

What is a polynomial representative for 
the coset associated to $[X]$ under Borel's isomorphism? This question is well-studied for
Schubert varieties, i.e., the $B$-orbit closures. Particularly nice representatives were discovered by A.~Lascoux-M.-P.~Sch\"{u}tzenberger \cite{Lascoux.Schutzenberger} in type $A$. 
In their work, \emph{Schubert polynomials} are defined starting with a choice of representative of the unique closed orbit 
(a point). In general, the polynomials
are obtained recursively using {\bf divided difference operators} defined by
$\partial_i(f):=\frac{f-s_i(f)}{x_i-x_{i+1}}$. Well-definedness 
of these polynomials follows from the fact these operators
define a representation of the symmetric group $S_n$.

For $K$-orbit closures, one can still use divided differences to obtain representatives for any orbit closure starting
with a representative for the class of each closed orbit.  
However, well-definedness is somewhat more subtle.  For instance, in the case of $(GL_{p+q},GL_p\times GL_q)$:
\begin{itemize}
\item[(i)] There are multiple closed orbits. 
\item[(ii)] Two saturated paths in the weak order joining the same two elements
are not necessarily labelled by reduced words of the same permutation. 
\end{itemize}
In \cite{WyserYong} we exhibit well-definedness of a choice of
representatives.  In contrast, for 
the cases studied in this paper, (i) is obviated since in each case $K$
does have a unique closed orbit.  However, (ii) remains true. Consequently, unlike the Schubert setting, the operators associated
to the two saturated paths may not be equal.

\subsection{Weak order, divided differences, and orbit parametrizations}\label{sec:weak-orders}
Let $G$ be a reductive group, and fix a maximal torus
and Borel subgroup $T \subseteq B$, respectively.  Let $\Delta=\Delta(G,T)$ denote the system of simple roots corresponding
to $B$.  Let $W=N_G(T)/T$ be the Weyl group, generated by simple reflections $\{s_{\ga} \mid \alpha \in \Delta\}$.  Let
$K$ be a spherical subgroup of $G$.
Given a $K$-orbit closure $Y$ and a simple reflection $s_{\ga}$, we define $s_{\ga} \cdot Y$ 
to be the $K$-orbit closure $\pi_{\ga}^{-1}(\pi_{\ga}(Y))$, where $\pi_{\ga}: G/B \rightarrow G/P_{\ga}$ is the
natural projection.  (Here, $P_{\ga} = B \cup Bs_{\ga}B$ is the standard minimal parabolic subgroup of type $\ga$.)
For $w \in W$, let $w = s_{\alpha_1} \hdots s_{\alpha_l}$ be a reduced expression for $w$, and define
$ w \cdot Y = s_{\alpha_1} \cdot (s_{\alpha_2} \cdot \hdots \cdot (s_{\alpha_l} \cdot Y) \hdots )$.
The resulting $K$-orbit closure is independent of the choice of reduced expression of $w$.  The \textbf{weak order} on
the set of $K$-orbit closures on $G/B$ is defined by $Y \leq Y'$ if and only if $Y' = w \cdot Y$ for some $w \in W$.

By \cite{Richardson-Springer}, $\pi_{\ga}|_Y$ is either birational or $2$-to-$1$ over its image.
Our weak order Hasse diagrams will use
a {\bf solid edge} labelled by $\ga$ to connect any $Y$ to $s_{\ga} \cdot Y \neq Y$ in the former case and a
{\bf dashed edge} in the latter case. (In general, multiple $\alpha$-labels can occur on an edge.) These are the conventions of \cite{Wyser-13b}, and correspond to the use of single and double edges, respectively, in \cite{Brion-01}. 

We need known parametrizations of the orbit sets for the two cases of this paper, as well as a concrete description 
of their weak orders. The reader may
refer to \cite{Wyser-13b} and the references therein, particularly \cite{Richardson-Springer}. First, for $(G,K)=(GL_n,O_n)$, the $K$-orbits are indexed by involutions $\pi \in S_n$ (i.e., $\pi^2 = \id$). Let $\caI(n)$ denote these involutions.
The {\bf Bruhat order} on $\caI(n)$, corresponding to containment of orbit closures, is the restriction of the inverse Bruhat
order on $S_n$.  Thus the unique closed orbit corresponds to the long element $w_0$, and the dense orbit to the
identity.  The weak order on $\caI(n)$ is generated by relations $\pi \prec s_i \cdot \pi \neq \pi$, where
$s_i$ is a simple transposition, and $s_i \cdot \pi$ is defined by:
\begin{enumerate}
	\item[(a)] If $\ell(s_i \pi) > \ell(\pi)$, then $s_i \cdot \pi = \pi$;
	\item[(b)] Otherwise, if $s_i \pi s_i \neq \pi$, then $s_i \cdot \pi = s_i \pi s_i$.  The edge connecting $\pi$
to $s_i \cdot \pi$ is solid.
	\item[(c)] Else, $s_i \cdot \pi = s_i \pi$.  The edge connecting $\pi$ to $s_i \cdot \pi$ is dashed.
\end{enumerate}

In case (b) above, if $\caY_{\pi}$ is the $K$-orbit closure corresponding to the
involution $\pi$, we have
$ [\caY_{s_i \cdot \pi}] = \partial_i [\caY_{\pi}]$,
both in ordinary and $S$-equivariant cohomology.  In case (c), we have
$[\caY_{s_i \cdot \pi}] = \frac{1}{2} \partial_i [\caY_{\pi}]$.
(We are mildly
abusing the notation $\partial_i$ and $\frac{1}{2}\partial_i$ to mean the cohomological push-pull operators
that correspond to the polynomial operators.) 

For $(G,K) = (GL_{2n},Sp_{2n})$, the $K$-orbits are indexed by fixed point-free involutions in $S_{2n}$;  
let $\caI_{\text{fpf}}(2n)$ denote the set of such involutions.
The
{\bf Bruhat order} on $\caI_{\text{fpf}}(2n)$ is also 
the restriction of the inverse Bruhat order on $S_{2n}$, with the unique closed orbit corresponding to $w_0$, and the dense orbit corresponding to $(1,2)(3,4) \hdots (2n-1,2n)$.  The weak order on $\caI_{\text{fpf}}(2n)$
is defined by
\begin{enumerate}
	\item[(a')] If $\ell(s_i \pi) > \ell(\pi)$, or if $s_i \pi s_i = \pi$, then $s_i \cdot \pi = \pi$;
	\item[(b')] Else, $s_i \cdot \pi = s_i \pi s_i$.  The edge connecting $\pi$ to $s_i \cdot \pi$ is solid.
\end{enumerate}

Let $\frX_{\pi}$ be the $K$-orbit closure corresponding to $\pi$. In case (b'), 
$[\frX_{s_i \cdot \pi}] = \partial_i [\frX_{\pi}]$, 
in both ordinary and $S$-equivariant cohomology.

\subsection{The $\Upsilon$-polynomials and their well-definedness}
Let
\begin{equation}
\label{eqn:GOeven}
\Upsilon_{w_0; (GL_n,O_n)} = \displaystyle\prod_{1 \leq i \leq j \leq n-i} (x_i + x_j).
\end{equation}
When $n$ is even, this is both the ordinary and $S$-equivariant cohomology
representative of $[\caY_{w_0}]$ from \cite[Proposition~2.5]{Wyser-13b}. 
On the other hand, when $n$ is odd,  the representative \eqref{eqn:GOeven}
is actually different than that of \cite[Proposition 2.1]{Wyser-13b}.
We will prove the latter's correctness
via Proposition \ref{prop:k-formula-o}, which gives the finer $K$-theoretic result; see Remark~\ref{rmk:k-to-c}.

For all $\pi\in \caI(n)$, a representative for $[\caY_{\pi}]$ can be
computed starting from $\Upsilon_{w_0; (GL_{n},O_{n})}$  and applying a sequence of (possibly $\frac{1}{2}$-scaled) 
divided difference operators along a saturated path from $w_0$ to $\pi$, as indicated in Section~1.2. That is, suppose 
\[\pi=s_{i_1} \cdot (s_{i_2} \cdot \hdots \cdot (s_{i_l} \cdot w_0) \hdots ).\] Then define 
\[\Upsilon_{\pi; (GL_n,O_n)}:={\widetilde \partial}_{i_1}{\widetilde \partial}_{i_2}\cdots{\widetilde \partial}_{i_l}\Upsilon_{w_0; (GL_{n},O_{n})},\]
where ${\widetilde \partial}_i$ equals $\partial_i$ if $i$ corresponds to a
single edge in the path, and $\frac{1}{2}\partial_i$
otherwise.

For the case $(GL_{2n},Sp_{2n})$ define
\begin{equation}
\label{eqn:GS}
\Upsilon_{w_0; (GL_{2n},Sp_{2n})}= \displaystyle\prod_{1 \leq i < j \leq 2n-i} 
(x_i + x_j).
\end{equation}
This is the representative of $[\frX_{w_0}]$ from \cite[Proposition~2.7]{Wyser-13b}.  Given any $\pi \in \mathcal{I}_{\text{fpf}}(2n)$, we can similarly compute a representative for $[\frX_{\pi}]$ using divided difference operators along a chosen saturated path in the weak order.  In this manner, we similarly define polynomials $\Upsilon_{\pi; (GL_{2n},Sp_{2n})}$.

We emphasize that the same $\Upsilon$ polynomials represent
the classes of orbit closures in both ordinary and $S$-equivariant cohomology.  In the equivariant case, they only vacuously involve the extra set of ``equivariant''
$y$-variables; cf. (\ref{eqn:BorelKT}). By contrast, the \emph{double 
Schubert polynomials}, which represent the $T$-equivariant
cohomology classes of Schubert varieties, actually involve these additional ``$y$'' variables. 
This is also true of the $T$-equivariant
representatives for $GL_p\times GL_q$-orbit closures from \cite{WyserYong}.

Our point is to establish that the 
\textbf{$\Upsilon$-polynomials} above are actually well-defined:

\begin{Theorem}
\label{thm:self-consistent}
$\Upsilon_{\pi,(GL_n,O_n)}$ and $\Upsilon_{\pi,(GL_{2n},Sp_{2n})}$ are independent of the path in weak order used to compute them. In addition, each $\Upsilon$-polynomial is a nonnegative linear combination of (single) Schubert polynomials ${\mathfrak S}_w(x_1,\ldots,x_n)$ for $w\in S_n$, and therefore is in ${\mathbb Z}_{\geq 0}[x_1,\ldots,x_n]$.
\end{Theorem}

The second statement of the theorem implies that
in ordinary cohomology, these polynomials can also be deduced by a combinatorial formula of 
M.~Brion \cite[Theorem 1.5]{Brion-98}. However, it seems nontrivial to 
explicitly relate the two descriptions of the representatives.
We comment on this further in Section~5.

\begin{Example}
The weak order of $(GL_4,O_4)$ is pictured in Figure \ref{fig:type-a-orthogonal-2}.  The closed orbit corresponds to $w_0=(1,4)(2,3)$, and its class is represented by
\[\Upsilon_{w_0; (GL_4,O_4)} = 4x_1x_2(x_1+x_2)(x_1+x_3).\]
\begin{figure}[h!]
	\centering
	\includegraphics[scale=0.5]{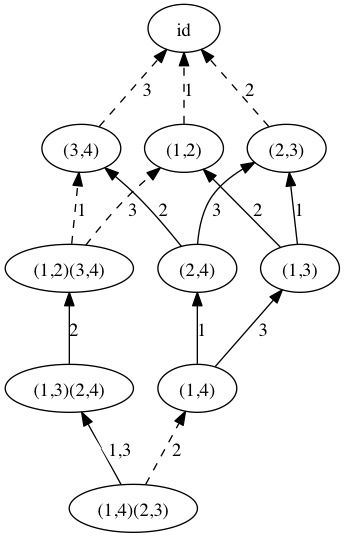}
	\caption{Labelled Hasse diagram for weak order of $(GL_4,O_4)$}\label{fig:type-a-orthogonal-2}
\end{figure}
Consider the orbit for $\pi = (3,4)$. One path from $w_0$ to $\pi$ is labelled
by $s_2s_1s_2$, while another is labelled by $s_1s_2s_3$.  Taking into account
dashed edges, these paths correspond respectively to 
$\frac{1}{2} \partial_2 \partial_1 \partial_2$ and $\frac{1}{2} \partial_1 \partial_2 \partial_3$. 
Although $\frac{1}{2} \partial_2 \partial_1 \partial_2\neq \frac{1}{2} \partial_1 \partial_2 \partial_3$, we have
\[ \frac{1}{2} \partial_2\partial_1\partial_2\Upsilon_{w_0; (GL_4,O_4)} = 
\frac{1}{2} \partial_1\partial_2\partial_3\Upsilon_{w_0; (GL_4,O_4)} = 
 2(x_1+x_2+x_3), \]
in agreement with Theorem~\ref{thm:self-consistent}.

The full table of $\Upsilon$ polynomials for the pair $(GL_4,O_4)$ is given in Table \ref{tab:type-a-o4}.\qed
\end{Example}

\begin{table}[h]

	\begin{tabular}{|l|l|}
		\hline
		Involution $\pi$ & $\Upsilon_{\pi; (GL_4,O_4)}$ \\ \hline
		$(1,4)(2,3)$ & $4x_1x_2(x_1+x_2)(x_1+x_3)$ \\ \hline
		$(1,3)(2,4)$ & $4x_1x_2(x_1+x_2)$ \\ \hline
		$(1,4)$ & $2x_1(x_1+x_2)(x_1+x_3)$ \\ \hline
		$(1,2)(3,4)$ & $4x_1(x_1+x_2+x_3)$ \\ \hline
		$(1,3)$ & $2x_1(x_1+x_2)$ \\ \hline
		$(2,4)$ & $2(x_1+x_2)(x_1+x_2+x_3)$ \\ \hline
		$(1,2)$ & $2x_1$ \\ \hline
		$(3,4)$ & $2(x_1+x_2+x_3)$ \\ \hline
		$(2,3)$ & $2(x_1+x_2)$ \\ \hline
		id & $1$ \\ 
		\hline
	\end{tabular}
	\caption{Polynomial representatives for $(GL_4,O_4)$}\label{tab:type-a-o4}
\end{table}

\begin{Example}\label{ex:gl6sp6}
Consider the pair $(GL_6,Sp_6)$.  The weak order on $\mathcal{I}_{\text{fpf}}(6)$ is depicted in Figure~\ref{fig:type-a-symplectic-2}.
\begin{figure}[h!]

	\centering
	\includegraphics[scale=0.5]{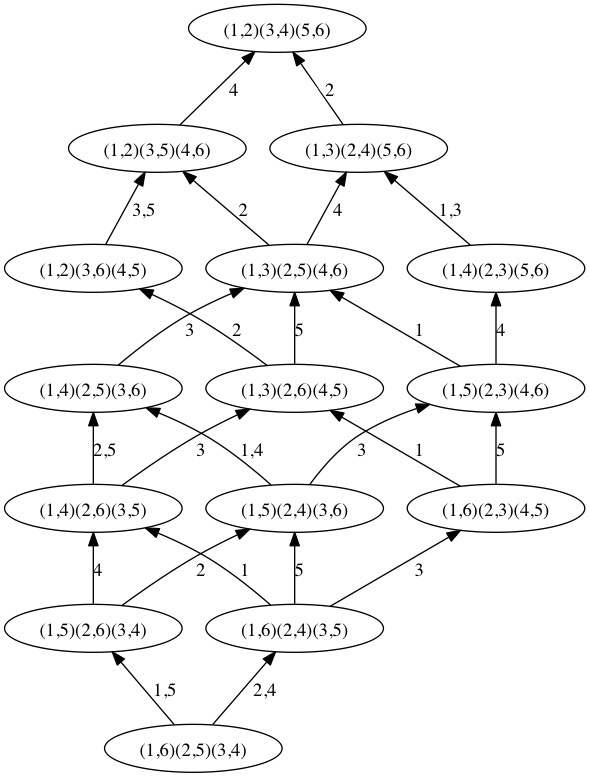}
	\caption{Labelled Hasse diagram for weak order of $(GL_6,Sp_6)$}\label{fig:type-a-symplectic-2}
\end{figure}
The class of the closed orbit, indexed by $w_0=(1,6)(2,5)(3,4)$, is represented by
\[ \Upsilon_{w_0; (GL_6,Sp_6)} = (x_1+x_2)(x_1+x_3)(x_1+x_4)(x_1+x_5)(x_2+x_3)(x_2+x_4). \]
There are a couple of $2$-step paths from $w_0$ to $\pi = (1,5)(2,4)(3,6)$.  One path is labelled by $s_2s_1$, while another is labelled by $s_5s_2$.  We compute a representative for
$[\frX_{\pi}]$ by applying either $\partial_2  \partial_1$ or $\partial_5  \partial_2$ to
$\Upsilon_{w_0; (GL_6,Sp_6)}$.  Since $s_2s_1 \neq s_5s_2$, $\partial_2  \partial_1\neq \partial_5  \partial_2$.  Nevertheless, 
\[ \partial_2\partial_1\Upsilon_{w_0; (GL_6,Sp_6)} = \partial_5\partial_2\Upsilon_{w_0; (GL_6,Sp_6)} = 
 (x_1+x_2)(x_1+x_3)(x_1+x_4)(x_2+x_3), \]
again agreeing with Theorem~\ref{thm:self-consistent}.

The full table of $\Upsilon$ polynomials for the pair $(GL_6,Sp_6)$ is given in Table \ref{tab:type-a-sp6}.\qed
\end{Example}

\begin{table}[h]
	
	\begin{tabular}{|c|l|}
		\hline
		Involution $\pi$ & $\Upsilon_{\pi; (GL_6,Sp_6)}$ \\ \hline
		$(1,6)(2,5)(3,4)$ & $(x_1+x_2)(x_1+x_5)(x_1+x_3)(x_1+x_4)(x_2+x_3)(x_2+x_4)$ \\ \hline
		$(1,5)(2,6)(3,4)$ & $(x_1+x_2)(x_1+x_3)(x_1+x_4)(x_2+x_3)(x_2+x_4)$ \\ \hline
		$(1,6)(2,4)(3,5)$ & $(x_1+x_2)(x_1+x_5)(x_1+x_3)(x_1+x_4)(x_2+x_3)$ \\ \hline
		$(1,4)(2,6)(3,5)$ & $(x_1+x_2)(x_1+x_3)(x_2+x_3)(x_1+x_2+x_4+x_5)$ \\ \hline		
		$(1,5)(2,4)(3,6)$ & $(x_1+x_2)(x_1+x_3)(x_1+x_4)(x_2+x_3)$ \\ \hline
		$(1,6)(2,3)(4,5)$ & $(x_1+x_2)(x_1+x_5)(x_1+x_3)(x_1+x_4)$ \\ \hline
		$(1,4)(2,5)(3,6)$ & $(x_1+x_2)(x_1+x_3)(x_2+x_3)$ \\ \hline
		$(1,3)(2,6)(4,5)$ & $(x_1+x_2)(x_1^2+x_2^2+\displaystyle\sum_{1 \leq i < j \leq 5} x_ix_j)$ \\ \hline
		$(1,5)(2,3)(4,6)$ & $(x_1+x_2)(x_1+x_3)(x_1+x_4)$ \\ \hline		
		$(1,2)(3,6)(4,5)$ & $(x_1+x_2+x_3+x_4)(x_1+x_2+x_3+x_5)$ \\ \hline
		$(1,3)(2,5)(4,6)$ & $(x_1+x_2)(x_1+x_2+x_3+x_4)$ \\ \hline
		$(1,4)(2,3)(5,6)$ & $(x_1+x_2)(x_1+x_3)$ \\ \hline
		$(1,2)(3,5)(4,6)$ & $x_1+x_2+x_3+x_4$ \\ \hline
		$(1,3)(2,4)(5,6)$ & $x_1+x_2$ \\ \hline
		$(1,2)(3,4)(5,6)$ & $1$ \\
		\hline
	\end{tabular}
\caption{Polynomial representatives for $(GL_6,Sp_6)$}\label{tab:type-a-sp6}
\end{table}

\subsection{Stability of the $\Upsilon$-polynomials}\label{sec:intro-stability}
The natural inclusion of $GL_n$ into $GL_N$ ($N \geq n$) 
sends an $n \times n$ matrix $A$ to the
$N \times N$ block matrix whose upper left $n \times n$ corner is $A$, whose lower-right $(N-n) \times (N-n)$ corner is
the identity, and whose off-diagonal blocks are zero.  This inclusion sends the standard Borel of $GL_n$ into the
standard Borel of $GL_N$, and thereby induces an inclusion of flag varieties $Fl_n \hookrightarrow Fl_N$.

First, consider the pair $(GL_n,O_n)$.  Define $\iota:S_n \hookrightarrow S_N$ to send a permutation $\pi$ to the unique permutation 
$\pi':=\iota(\pi)$ satisfying
$\pi'(i) = \pi(i)$ for $i \leq n$, and $\pi'(i) = i$ for $i > n$.  Clearly, if $\pi \in \caI(n)$, then $\pi' \in \caI(N)$.
It follows, for example, from the set-theoretic description of $K$-orbit closures \cite{Wyser-13b}, that
the image of $\caY_{\pi}$
under the aforementioned inclusion $Fl_n \hookrightarrow Fl_N$, is the $O_N$-orbit closure $\caY_{\pi'}$.  

For the case $(GL_{2n},Sp_{2n})$, if $\pi \in \caI_{\text{fpf}}(2n)$, consider the
map $\iota_{\text{fpf}}: S_{2n} \hookrightarrow S_{2N}$ which sends $\pi \in S_{2n}$ to the unique permutation
$\pi':=\iota_{\text{fpf}}(\pi) \in S_{2N}$ which coincides with $\pi$ on $\{1,\hdots,2n\}$, and which transposes $i$ and
$i+1$ for $i=2n+1,2n+3,\hdots,2N-3,2N-1$.  Clearly, $\iota_{\text{fpf}}\in \caI_{\text{fpf}}(2N)$. The set-theoretic description of the orbit closures from  \cite{Wyser-13b} implies that the image of $\frX_{\pi} \subseteq Fl_{2n}$
under the inclusion $Fl_{2n} \hookrightarrow Fl_{2N}$ is $\frX_{\pi'}$.  

We view the following theorem as an analogue of the stability property for Schubert polynomials, which states that
$\Schub_{\iota(w)}=\Schub_{w}$ (see, e.g., \cite[Corollary~2.4.5]{Manivel}):
\begin{Theorem}\label{thm:stability}
For $n \leq N$,
$\Upsilon_{\iota(\pi); (GL_N,O_N)} = \Upsilon_{\pi; (GL_n,O_n)}$ and
$\Upsilon_{\iota_{\text{fpf}}(\pi); (GL_{2N},Sp_{2N})} = \Upsilon_{\pi; (GL_{2n},Sp_{2n})}$.
\end{Theorem}

We are unaware of any analogous stability property of the representatives considered in \cite{WyserYong}
for the pair $(GL_{p+q},GL_p \times GL_q)$.

\subsection{Organization} In Section~2, we describe some $K$-theoretic extensions of the theorems of this section. These provide
the first $K$-theoretic results for the symmetric pairs considered
in this paper, complementing those for the $K=GL_p\times GL_q$ case
of \cite{WyserYong}. We have complete analogues for the case 
$(GL_{2n},Sp_{2n})$, using \emph{Demazure operators}. However, for $(GL_n,O_n)$, we only provide a
formula for the $K$-class of the closed orbit. In the latter case, we demonstrate by example 
that Demazure operators cannot be similarly used
to make computations. 
In Section~3, we give combinatorial proofs of Theorems~\ref{thm:self-consistent}
and~\ref{thm:stability}.
Section~4 gives the proofs of the results of Section~2, using
equivariant localization arguments combined with a self-intersection formula of R.~W.~Thomason. In Section~5, we present some final remarks.

\section{$K$-theoretic extensions}\label{sec:intro-k-theory}
\subsection{Results}
For the pair $(GL_{2n},Sp_{2n})$, we give complete extensions of Section~1's results to ($S$-equivariant)
$K$-theory, the Grothendieck ring for the category of ($S$-equivariant) locally free sheaves on $GL_{2n}/B$.  

The ordinary
$K$-theory ring can be realized concretely as
\begin{equation}
\label{eqn:BorelK}
K^0(GL_{2n}/B) \cong \Z[x_1,\hdots,x_{2n}]/I, 
\end{equation}
where $I$ is the ideal generated by $e_d(x_1,\hdots,x_{2n}) - \binom{2n}{d}$ for $1 \leq d \leq 2n$, with $e_d$ the
elementary symmetric polynomial of degree $d$; cf.~\cite[Section 2.3]{Knutson.Miller:annals}.

The $S$-equivariant $K$-theory can be realized as a quotient of a (Laurent) polynomial ring in two sets of variables,
namely as
\begin{equation}
\label{eqn:BorelKT}
 K_S^0(GL_{2n}/B) \cong \Z[x_1^{\pm 1},\hdots,x_{2n}^{\pm 1}][y_1^{\pm 1},\hdots,y_n^{\pm 1}]/J, 
\end{equation}
with $J$ generated by the differences $e_d(x_1,\hdots,x_{2n})-e_d(y_1,\hdots,y_n,y_n^{-1},\hdots,y_1^{-1})$.  This can be deduced from the description of $K_S^0(G/B)$ given in Section \ref{sec:background-k-theory}, for which a reference is \cite{Kostant-Kumar}.
The natural map from $K_S^0(G/B)$ to $K^0(G/B)$ which forgets the $S$-equivariant structure corresponds to setting
all $y_i$ equal to $1$.

Thus for $\pi \in \caI_{\text{fpf}}(2n)$, we seek a polynomial in the $x$ and $y$
variables which represents $[\caO_{\frX_{\pi}}]$, the class of the structure sheaf of the orbit closure $\frX_{\pi}$
(considered as a coherent sheaf on $GL_{2n}/B$) in $K_S^0(GL_{2n}/B)$. As with the results of Section~1, only the
$x$ variables actually appear in our representatives:

\begin{Theorem}\label{thm:k-formula-sp}
  $[\caO_{\frX_{w_0}}]$ is represented, in both ordinary and $S$-equivariant $K$-theory, by the polynomial
\[ \Upsilon_{w_0, (GL_{2n},Sp_{2n})}^K = \displaystyle\prod_{1 \leq i < j \leq 2n-i} (1 - x_ix_j). \]
\end{Theorem}

The Demazure operator (or ``isobaric divided difference operator'') $D_i$ is a $K$-theoretic analogue of the divided
difference operator $\partial_i$.  It is defined by:
\[ D_i(f) = \dfrac{x_{i+1}f - x_i s_i(f)}{x_{i+1}-x_i} = -\partial_i(x_{i+1}f). \]

If one starts with $\Upsilon_{w_0, (GL_{2n},Sp_{2n})}^K$ and applies a sequence of Demazure operators corresponding to a path from $w_0$ to $\pi$ in $\caI_{\text{fpf}}(2n)$, the result $\Upsilon_{\pi, (GL_{2n},Sp_{2n})}^K$ represents $[\caO_{\frX_{\pi}}]$ in both ordinary $K$-theory (\ref{eqn:BorelK}) and $S$-equivariant $K$-theory (\ref{eqn:BorelKT}).  
This is justified in Section \ref{sec:div-diff-probs}. As in Section~1,
we show that the representative is independent of the sequence of Demazure operators applied.

\begin{Theorem}\label{thm:well-defined-stable-k-theory}
The polynomial $\Upsilon_{\pi, (GL_{2n},Sp_{2n})}^K$ is independent of the choice of path in weak order used to compute it.
The $\Upsilon^K$-polynomials are stable with respect to the inclusion $\iota_{\text{fpf}}: \caI_{\text{fpf}}(2n) \hookrightarrow \caI_{\text{fpf}}(2N)$ for any $N \geq n$.
\end{Theorem}

The $\Upsilon^K$-polynomials for the case $(GL_6,Sp_6)$ are given in Table \ref{tab:type-a-sp6-k}.

\begin{table}[h]

	\begin{tabular}{|c|l|}
		\hline
		Involution $\pi$ & $\Upsilon_{\pi; (GL_6,Sp_6)}$ \\ \hline
		$(1,6)(2,5)(3,4)$ & $(1-x_1x_2)(1-x_1x_3)(1-x_1x_4)(1-x_1x_5)(1-x_2x_3)(1-x_2x_4)$ \\ \hline
		$(1,5)(2,6)(3,4)$ & $(1-x_1x_2)(1-x_1x_3)(1-x_1x_4)(1-x_2x_3)(1-x_2x_4)$ \\ \hline
		$(1,6)(2,4)(3,5)$ & $(1-x_1x_2)(1-x_1x_3)(1-x_1x_4)(1-x_1x_5)(1-x_2x_3)$ \\ \hline
		$(1,4)(2,6)(3,5)$ & $(1-x_1x_2)(1-x_1x_3)(1-x_2x_3)(1-x_1x_2x_4x_5)$ \\ \hline		
		$(1,5)(2,4)(3,6)$ & $(1-x_1x_2)(1-x_1x_3)(1-x_1x_4)(1-x_2x_3)$ \\ \hline
		$(1,6)(2,3)(4,5)$ & $(1-x_1x_2)(1-x_1x_3)(1-x_1x_4)(1-x_1x_5)$ \\ \hline
		$(1,4)(2,5)(3,6)$ & $(1-x_1x_2)(1-x_1x_3)(1-x_2x_3)$ \\ \hline
		$(1,3)(2,6)(4,5)$ & $(1-x_1x_2)(1\!-\!x_1x_2x_3x_4\!-\!x_1x_2x_3x_5\!-\!x_1x_2x_4x_5\!+\!x_1^2x_2x_3x_4x_5\!+\!x_1x_2^2x_3x_4x_5)$ \\ \hline
		$(1,5)(2,3)(4,6)$ & $(1-x_1x_2)(1-x_1x_3)(1-x_1x_4)$ \\ \hline		
		$(1,2)(3,6)(4,5)$ & $(1-x_1x_2x_3x_4)(1-x_1x_2x_3x_5)$ \\ \hline
		$(1,3)(2,5)(4,6)$ & $(1-x_1x_2)(1-x_1x_2x_3x_4)$ \\ \hline
		$(1,4)(2,3)(5,6)$ & $(1-x_1x_2)(1-x_1x_3)$ \\ \hline
		$(1,2)(3,5)(4,6)$ & $1-x_1x_2x_3x_4$ \\ \hline
		$(1,3)(2,4)(5,6)$ & $1-x_1x_2$ \\ \hline
		$(1,2)(3,4)(5,6)$ & $1$ \\
		\hline
	\end{tabular}
	\caption{$K$-theoretic polynomial representatives for $(GL_6,Sp_6)$}\label{tab:type-a-sp6-k}
\end{table}

For $(GL_n,O_n)$, we only have a formula for the class of the closed orbit:
\begin{Proposition}\label{prop:k-formula-o}
  $[\caO_{\caY_{w_0}}]$ is represented, in both ordinary and $S$-equivariant $K$-theory, by the polynomial
\[ \Upsilon_{w_0, (GL_n,O_n)}^K = \displaystyle\prod_{1 \leq i \leq j \leq n-i} (1 - x_ix_j). \]
\end{Proposition}

As explained in Section~2.2, one cannot expect Demazure operators to compute the $K$-theory representatives
for  $(GL_n,O_n)$. That said, we  computed tables of representatives for $n=3,4$ differently.  They appear in Tables \ref{tab:type-a-o3-k} and \ref{tab:type-a-o4-k}.  (Recall that the weak order for the case $n=4$ was given as Figure \ref{fig:type-a-orthogonal-2}.  The weak order for $n=3$ appears in Figure \ref{fig:type-a-orthogonal-1}.)

The representatives in Tables \ref{tab:type-a-o3-k} and \ref{tab:type-a-o4-k} were computed using a geometric perspective originally applied by A.~Knutson-E.~Miller \cite{Knutson.Miller:annals} to justify Schubert polynomials.
For a variety $X\subset GL_n/B$, consider the preimage $\pi^{-1}(X)\subset GL_n$ under the natural projection,
and $\overline{\pi^{-1}(X)}\subset {\rm Mat}_{n\times n}$.  If $X$ is a $K$-orbit closure, then $\overline{\pi^{-1}(X)}$ is stable under the action of $B_K \times B$, where $B_K = B \cap K$ is the standard Borel subgroup of $K$ (acting by multiplication on the left), and $B$ is the standard Borel of $GL_n$ (acting by inverse multiplication on the right).
Identifying
\[[{\mathcal O}_{\overline{\pi^{-1}(X)}}]_{B_K \times B} \in K^\circ_{B_K \times B}({\rm Mat}_{n\times n}) \mbox{\  with 
$[{\mathcal O}_{\overline{\pi^{-1}(X)}}]_{S \times T} \in K^\circ_{S \times T}({\rm Mat}_{n\times n})$}\]
(see \cite[Corollary 2.3.1 and Remark 2.3.3]{Knutson.Miller:annals}) uniquely picks out a polynomial representative for 
$[{\mathcal O}_X]\in K^{\circ}_S (GL_n/B)$.  The class $[{\mathcal O}_{\overline{\pi^{-1}(X)}}]_{S \times T}$ is computable as the multigraded $K$-polynomial of the ideal of $\overline{\pi^{-1}(X)}$, where the multigrading arises from the $S \times T$-action on ${\rm Mat}_{n\times n}$.

To apply this in the case of $X={\mathcal Y}_\pi$, we use the set-theoretic description of the orbit closures from \cite{Wyser-13b} to
deduce set-theoretically correct equations for $\pi^{-1}(X)\subset GL_{n}$.  These equations are scheme-theoretically
correct for the {\bf matrix $K$-orbit} $\overline{\pi^{-1}(X)}\subset M_{n\times n}$ provided they generate a prime ideal. We have computationally verified that this is indeed the
case for $n=3,4$.  The $K$-theoretic representatives of Tables \ref{tab:type-a-o3-k} and \ref{tab:type-a-o4-k} are the aforementioned $K$-polynomials, computed as numerators of multigraded Hilbert series.  All computations were carried out using {\tt Macaulay 2}.

Empirically, one notices that after the substitution $x_i\mapsto 1-x_i$, the representatives computed this way alternate in sign by degree.  When the $K$-orbits have rational singularities, this is expected; see the comments at the end of Section~4.  However, $O_n$-orbit closures on $GL_n/B$ do not have rational singularities in general.  Thus we are not aware of an explanation of this apparent alternation in sign, if in fact it holds in general.

\begin{figure}[h!]
	\centering
	\includegraphics[scale=0.5]{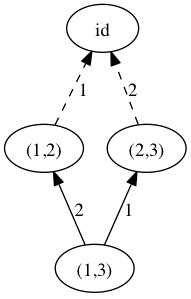}
	\caption{Labelled Hasse diagram for weak order of $(GL_3,O_3)$}\label{fig:type-a-orthogonal-1}
\end{figure}

\begin{table}[h]

	\begin{tabular}{|l|l|}
		\hline
		Involution $\pi$ & $K$-polynomial for matrix $K$-orbit \\ \hline
		$(1,3)$ & $(1-x_1^2)(1-x_1x_2)$ \\ \hline
		$(1,2)$ &  $1-x_1^2$ \\ \hline
		$(2,3)$ & $(1-x_1x_2)(1+x_1x_2)$ \\ \hline
		id & $1$ \\ 
		\hline
	\end{tabular}
	\caption{$K$-theoretic polynomial representatives for $(GL_3,O_3)$}\label{tab:type-a-o3-k}
\end{table}

\begin{table}[h]

	\begin{tabular}{|l|l|}
		\hline
		Involution $\pi$ & $K$-polynomial for matrix $K$-orbit \\ \hline
		$(1,4)(2,3)$ & $(1-x_1^2)(1-x_2^2)(1-x_1x_2)(1-x_1x_3)$ \\ \hline
		$(1,3)(2,4)$ & $(1-x_1^2)(1-x_2^2)(1-x_1x_2)$ \\ \hline
		$(1,4)$ & $(1-x_1^2)(1-x_1x_2)(1-x_1x_3)$ \\ \hline
		$(1,2)(3,4)$ & $(1-x_1^2)(1-x_1x_2x_3)(1+x_1x_2x_3)$ \\ \hline
		$(1,3)$ & $(1-x_1^2)(1-x_1x_2)$ \\ \hline
		$(2,4)$ & $(1-x_1x_2)(1+x_1x_2-x_1^2x_2x_3-x_1x_2^2x_3)$ \\ \hline
		$(1,2)$ & $1-x_1^2$ \\ \hline
		$(3,4)$ & $(1-x_1x_2x_3)(1+x_1x_2x_3)$ \\ \hline
		$(2,3)$ & $(1-x_1x_2)(1+x_1x_2)$ \\ \hline
		id & $1$ \\ 
		\hline
	\end{tabular}
	\caption{$K$-theoretic polynomial representatives for $(GL_4,O_4)$}\label{tab:type-a-o4-k}
\end{table}

\subsection{Demazure operators in $K$-theory}\label{sec:div-diff-probs}
Here we give a standard explanation of 
why Demazure operators are valid for $K$-theoretic computations for the pair $(GL_{2n},Sp_{2n})$. We also explain why we cannot use them to obtain representatives for $(GL_n,O_n)$. 

A $K$-orbit closure is {\bf multiplicity-free} if every saturated chain in weak order connecting it to the unique maximal orbit consists only of solid edges; cf. Section~1.2.
In particular, all orbit closures for $(GL_{2n},Sp_{2n})$ are of this type, since the weak order graphs for that pair contain no dashed edges at all. The following result is part of \cite[Theorem~6]{Brion-01}:

\begin{Theorem}[\cite{Brion-01}]\label{thm:brion}
 If $Y$ is any multiplicity-free $K$-orbit closure on $G/B$, then $Y$ has rational singularities.
\end{Theorem}

Consider any two $Sp_{2n}$-orbit closures $Y$ and $Y'$ on $GL_{2n}/B$, with $Y'$ covering $Y$ in the weak order.
They are connected by a solid edge, say with label $i$. As explained in Section \ref{sec:weak-orders}, this solid
edge indicates that the map $\pi: Y \rightarrow \pi(Y)$ (the
restriction of the natural map $\pi: G/B \rightarrow G/P_i$ to $Y$) is birational.  Now, by Theorem \ref{thm:brion},
both $Y$ and $Y'$ have rational singularities.  Thus $\pi(Y)$ has the same properties, since $Y'$ is a
$\mathbb{P}^1$-bundle over $\pi(Y)$.  In particular, $\pi(Y)$ is normal.  Zariski's Main Theorem \cite[Ch. III, \S 11]{Hartshorne} then implies that
\[\pi_* \mathcal{O}_Y \cong \mathcal{O}_{\pi(Y)}.\]  
On the other hand, if we consider the diagram
\[
\begin{tikzcd}
\widetilde{Y} \arrow{d}{r} \arrow{dr}{q = \pi \circ r} \\
Y \arrow{r}{\pi} & \pi(Y)
\end{tikzcd}
\]
(where $\widetilde{Y} \stackrel{r}{\longrightarrow} Y$ is a resolution of singularities for $Y$), then a straightforward
argument using the Leray spectral sequence
\[ R^i \pi_* (R^j r_* \mathcal{O}_{\widetilde{Y}}) \Rightarrow R^{i+j} q_* \mathcal{O}_{\widetilde{Y}} \]
implies that $R^i \pi_* \mathcal{O}_Y = 0$.  Thus in $K$-theory, we have that
\[ \pi_* [\mathcal{O}_Y] := [\pi_* \mathcal{O}_Y] + \displaystyle\sum_{i \geq 1} [R^i \pi_* \mathcal{O}_Y] = [\mathcal{O}_{\pi(Y)}]. \]
Then it follows that
\[ [\mathcal{O}_{Y'}] = [\mathcal{O}_{\pi^{-1}(\pi(Y))}] = \pi^*[\mathcal{O}_{\pi(Y)}] = (\pi^* \circ \pi_*)[\mathcal{O}_Y] = D_i [\mathcal{O}_Y], \]
since $\pi^* \circ \pi_*$ is precisely the operator $D_i$.

The above argument cannot always be applied
for the pair $(GL_n,O_n)$. There are two situations we discuss. The first occurs when instead of being connected by a solid edge, $Y$ and $Y'$ are connected by a dashed edge.  The second occurs when $Y$ and $Y'$ are connected by a solid edge, but $Y'$ is not normal.  In the first case, $\pi: Y \rightarrow \pi(Y)$ is no longer birational (rather, it has degree $2$), while in the second case, $\pi$ \textit{is} birational, but $\pi(Y)$ is now not normal.  In either event, Zariski's Main Theorem no longer applies, so 
we do not expect that $\pi_* \mathcal{O}_Y \cong \mathcal{O}_{\pi(Y)}$.  Actually, in both cases there is an injection
\[ \mathcal{O}_{\pi(Y)} \hookrightarrow \pi_* \mathcal{O}_Y,\]
but this map is not necessarily an isomorphism. Suppose the map's cokernel is $\mathcal{C}$.  Then in $K$-theory, 
\[[\pi_* \mathcal{O}_Y] = [\mathcal{O}_{\pi(Y)}] + [\mathcal{C}],\]
so when computing $\pi_*[\mathcal{O}_Y]$ there may be nontrivial corrections.

We now show that these correction terms are in fact nontrivial. We give an example of each, both for $(GL_4,O_4)$.  

\begin{Example}
First, consider the codimension-$1$ $O_4$-orbit closure $Y$ on $GL_4/B$ corresponding to the involution $(1,2)$.  It is connected to $Y'=G/B$ (the closure of the dense orbit, which corresponds to the identity) by a dashed edge with label $1$.
The ($S$-equivariant and ordinary) $K$-class of $Y$ is represented by $1-x_1^2$ (cf. Table \ref{tab:type-a-o4-k}). Since $Y'=GL_4/B$, its $K$-class is
represented by $1$.  However, when we apply $D_1$ (or, if we like, $\frac{1}{2} D_1$) to the class represented by $1-x_1^2$,
we get $1+x_1x_2$ (resp. $\frac{1}{2}(1+x_1x_2)$). One checks that neither $1+x_1x_2$ nor
$\frac{1}{2}(1+x_1x_2)$ equals $1$ (modulo $I$).  So applying the Demazure operator to $[\caO_{Y}]$ does not give $[\caO_{Y'}]$.
\qed
\end{Example}

\begin{Example}
Now suppose the edge is solid, but $Y'$ is not normal. For example, let $Y$ be the $O_4$-orbit closure on $GL_4/B$ corresponding to the involution $(1,3)(2,4)$, and let $Y'$ be the orbit closure corresponding to $(1,2)(3,4)$.  Then $Y$ is connected to $Y'$ by a solid edge with label $2$, and $Y'$ is not normal along the orbit closure corresponding to $(1,4)$ \cite[Corollary 4.4.4]{Perrin}.  (In fact, $Y'$ is reducible along $(1,4)$.)  Here we have (cf. Table \ref{tab:type-a-o4-k})
\[ [\mathcal{O}_Y] = (1-x_1^2)(1-x_2^2)(1-x_1x_2), \]
while
\begin{equation}
\label{eqn:simonecomp}
[\mathcal{O}_{Y'}] = (1-x_1^2)(1-x_1x_2x_3)(1+x_1x_2x_3). 
\end{equation}

Yet $D_2([\mathcal{O}_Y]) = (1-x_1^2)(1+x_2x_3-x_1x_2^2x_3-x_1x_2x_3^2)$, 
which is not equal (modulo $I$) to the representative (\ref{eqn:simonecomp}) of $[\mathcal{O}_{Y'}]$.  
So again, we see that $D_2([\caO_{Y}]) \neq [\caO_{Y'}]$.
\qed
\end{Example}

\section{Proofs of Theorems~\ref{thm:self-consistent} and~\ref{thm:stability}}
\subsection{Proof of Theorem~\ref{thm:self-consistent}} 
For any $n$, let 
\[{\mathcal L}_n \subseteq {\mathbb Q}[x_1,\ldots,x_n]\] be the span of monomials 
$x_1^{\alpha_1}\cdots x_n^{\alpha_n}$ where $\alpha_i\leq n-i$.  It is well-known that the single Schubert
polynomials $\{\Schub_w(X) \mid w \in S_n\}$ form a $\Z$-linear basis for ${\mathcal L}_n$.

Clearly $\Upsilon_{w_0;(GL_n,O_n)}$ is an element of ${\mathcal L}_n$.  Thus $\Upsilon_{w_0;(GL_n,O_n)}$
is a $\Z$-linear combination of single Schubert polynomials whose indexing permutations lie in $S_n$.  Since divided
difference operators send Schubert polynomials in $S_n$ to other Schubert polynomials in $S_n$, given an involution
$\pi \in S_n$, any polynomial representative for $[\caY_{\pi}]$ that one obtains via a sequence of divided
difference operators is again a sum of single Schubert polynomials indexed by elements of $S_n$.  Any two such
representatives certainly represent $[\caY_{\pi}]$.  But there can be only one polynomial representative for
$[\caY_{\pi}]$ which is a
$\Z$-linear combination of Schubert polynomials from $S_n$ --- namely, the one which uses precisely the
Schubert polynomials which correspond to the Schubert (cohomology) \textit{classes} $[X^w]$ appearing in the unique
expression for $[\caY_{\pi}]$ in the basis of Schubert classes.  Thus $\Upsilon_{\pi;(GL_n,O_n)}$ is well-defined, as
claimed.  The same argument applies verbatim to the pair $(GL_{2n},Sp_{2n})$.\qed

\subsection{Proof of Theorem~\ref{thm:stability}}
We give a detailed proof for the $(GL_{2n},Sp_{2n})$, and a sketch of the (similar) argument for the $(GL_n,O_n)$ case.
For fixed $N \geq n$, let $\iota_{\text{fpf}}$ be the map
$\caI_{\text{fpf}}(2n) \hookrightarrow \caI_{\text{fpf}}(2N)$ defined in Section
\ref{sec:intro-stability}.  Given $\pi \in \caI_{\text{fpf}}(2n)$ and any
weak-order path from $w_0$ to $\pi$, there is a
corresponding path from $\iota_{\text{fpf}}(w_0)$ to $\iota_{\text{fpf}}(\pi)$ in $\caI_{\text{fpf}}(2N)$ with
precisely the same edge labels.  Thus it suffices to prove that
\[ \Upsilon_{\iota_{\text{fpf}}(w_0); (GL_{2N},Sp_{2N})} = \Upsilon_{w_0; (GL_{2n},Sp_{2n})}. \]

Further, by induction, it is enough to show that for $N \geq 2$,
\[ \Upsilon_{\iota_{\text{fpf}}(w_0); (GL_{2N},Sp_{2N})} = \Upsilon_{w_0; (GL_{2(N-1)},Sp_{2(N-1)})}. \]

\begin{Lemma}\label{lem:symplectic-chain}
Let $\iota_{\text{fpf}}: \caI_{\text{fpf}}(2N-2) \hookrightarrow \caI_{\text{fpf}}(2N)$ be as in Section
\ref{sec:intro-stability}.  In $\caI_{\text{fpf}}(2N)$, there is a
path from $w_0$ to $\iota_{\text{fpf}}(w_0)$ with edges labelled $1,2,\hdots,2N-2$ (starting at the bottom and moving up).
\end{Lemma}
\begin{proof}
It is clear that the conjugation action (b') which defines the weak order (cf. Section \ref{sec:weak-orders}) of the
simple reflection $s_i$ on
any fixed point-free involution simply interchanges the positions of $i$ and $i+1$
within its cycle notation.  (If $i$ and $i+1$ are interchanged by the involution and appear within the same $2$-cycle,
then $s_i \cdot \pi$ is simply $\pi$, as in (a').)  So starting with
$w_0 \in S_{2N}$, where each $i=1,\hdots,N$ is paired with $2N+1-i$ in the cycle notation, we simply show that
consecutively acting by $s_1,\hdots,s_{2N-2}$ in this way gives the fixed point-free involution whose cycle notation
pairs $i$ with $2N-1-i$ for $i=1,\hdots,N-1$ (and thus, by necessity, pairs $2N-1$ with $2N$).

For such an $i$, note that the index $i$ is only moved by the reflections $s_{i-1}$, and then $s_i$.  Whatever the
action of
$s_{i-1}$, the action of $s_i$ must result in $i$ being paired with whatever $i+1$ is initially paired with, namely
$2N-i$, since when $s_i$ acts, neither $i+1$ nor its initial partner $2N-i$ have been affected.
The index $i$ is unaffected thereafter, but its new partner becomes $2N-1-i$ upon the action of $s_{2N-1-i}$.  Beyond
this point, neither $i$ nor $2N-1-i$ are affected again.  Thus $i$ is paired with $2N-1-i$, as claimed.
\end{proof}

Recall the polynomial
\[ \Upsilon_{w_0; (GL_{2N},Sp_{2N})} = \displaystyle\prod_{1 \leq i < j \leq 2N-i} (x_i+x_j) \]
from displayed equation \eqref{eqn:GS}, our chosen representative for the equivariant cohomology class of the closed
orbit $\frX_{w_0}$.  Call this polynomial $\Upsilon$ for short, and consider the effect of applying to it the divided
difference operators $\partial_1,\hdots,\partial_{2N-2}$, in that order.  Note that there are $N-1$ factors $x_i+x_j$
($i<j$) of $\Upsilon$ for which $i+j=2N$, namely $x_1+x_{2N-1},\hdots,x_{N-1}+x_{N+1}$.  There are also $N-1$ factors
$x_i+x_j$ ($i<j$) for which $i+j=2N-1$, those being $x_1+x_{2N-2},\hdots,x_{N-1}+x_N$.  The following establishes
Theorem~\ref{thm:stability} in the symplectic case:
\begin{Lemma}\label{lem:symplectic-divided-difference-sequence}
  When computing \[(\partial_{2N-2} \circ \partial_{2N-3} \circ \hdots \circ \partial_2 \circ \partial_1)(\Upsilon),\]
each of the first $N-1$ operators applied removes a linear factor $x_i+x_j$ with $i+j=2N$.  They are removed
in the order $x_1+x_{2N-1},\hdots,x_{N-1}+x_{N+1}$.  Each of the next $N-1$ operators applied 
removes a linear factor $x_i+x_j$ with $i+j=2N-1$.  They are removed in the order $x_{N-1}+x_N,\hdots,x_1+x_{2N-2}$.

  As a result, we have
  \[ \Upsilon_{\iota_{\text{fpf}}(w_0); (GL_{2N},Sp_{2N})} = \displaystyle\prod_{1 \leq i < j \leq 2N-2-i} (x_i+x_j) = 
    \Upsilon_{w_0; (GL_{2N-2},Sp_{2N-2})}.
  \]
\end{Lemma}
\begin{proof}
 We prove the first statement, on the application of the first $N-1$ operators, by induction.  (The proof of the second
statement, on application of the last $N-1$ operators, is exactly the same.)  First, note that
$\Upsilon^{(1)} := \Upsilon/(x_1+x_{2N-1})$ is symmetric in the variables $x_1$ and $x_2$.  Thus 
  \[ \partial_1(\Upsilon) = \dfrac{\Upsilon-s_1(\Upsilon)}{x_1-x_2} = 
      \dfrac{(x_1+x_{2N-1})\Upsilon^{(1)} - (x_2+x_{2N-1})\Upsilon^{(1)}}{x_1-x_2} = \Upsilon^{(1)}. \]
Now, suppose for $i < N-1$, after applying $\partial_i \circ \hdots \circ \partial_1$, we have
$\Upsilon^{(i)} := \Upsilon/\prod_{j=1}^i (x_j+x_{2N-j})$. 
Then $\Upsilon^{(i+1)} := \Upsilon/(x_{i+1}+x_{2N-1-i})$ is necessarily symmetric in the variables $x_{i+1}$ and $x_{i+2}$.
This implies that 
  \[ \partial_{i+1}(\Upsilon^{(i)}) = \dfrac{\Upsilon^{(i)}-s_{i+1}(\Upsilon^{(i)})}{x_{i+1}-x_{i+2}} = 
      \dfrac{(x_{i+1}+x_{2N-1-i})\Upsilon^{(i+1)} - (x_{i+2}+x_{2N-1-i})\Upsilon^{(i+1)}}{x_{i+1}-x_{i+2}} = 
\Upsilon^{(i+1)}. \]
By induction, the factors $x_1+x_{2N-1},\hdots,x_{N-1}+x_{N+1}$ are removed in the order claimed.
\end{proof}

The argument for the orthogonal case is very similar. We abbreviate it, leaving the straightforward modifications to the reader. By an identical argument to the above, it suffices to show that when
$w_0 \in \caI(N)$ is embedded into $\caI(N+1)$ as $\iota(w_0)$ ($\iota$ the embedding of $\caI(N)$ into
$\caI(N+1)$ defined in Section \ref{sec:intro-stability}), we have 
$\Upsilon_{\iota(w_0); (GL_{N+1},O_{N+1})} = \Upsilon_{w_0; (GL_N, O_N)}$.

Similar to the proof of Lemma \ref{lem:symplectic-chain}, one shows that in $\caI(N+1)$,
there is a path from $w_0$ to $\iota(w_0)$ with edge labels
$\lfloor (N+1)/2 \rfloor, \lfloor (N+1)/2 \rfloor + 1,\hdots,N$ when we start from the bottom of the weak order and
move up.  If $N+1$ is odd, all edges on this path are solid.  If $N+1$ is even, then the first edge (that
labelled $(N+1)/2$) is dashed, while the rest are solid.  

An analogue of Lemma \ref{lem:symplectic-divided-difference-sequence} shows
that the sequence of divided difference operators corresponding to this chain removes linear factors from 
$\Upsilon_{w_0; (GL_{N+1},O_{N+1})}$ predictably, leading to  $\Upsilon_{w_0; (GL_N,O_N)}$.
In the case where $N+1$ is odd, the argument is identical to the proof of Lemma
\ref{lem:symplectic-divided-difference-sequence}, with each operator stripping off one linear factor $x_i+x_j$ with
$i<j$ and $i+j=N+1$, in the order $x_{N/2}+x_{N/2+1},x_{N/2-1}+x_{N/2+2},\hdots,x_1+x_N$.  If $N+1$ is even, the first operator
$\frac{1}{2} \partial_{(N+1)/2}$ removes the factor $x_{(N+1)/2} + x_{(N+1)/2} = 2x_{(N+1)/2}$, and the remaining
operators remove the factors $x_{(N+1)/2-1}+x_{(N+1)/2+1},\hdots,x_1+x_N$, in that order.

\section{Background and proofs of the $K$-theory results of Section~2}
\subsection{Notation and conventions}\label{sec:conventions}
$G$ will be the general linear group over $\C$, $B$ its Borel subgroup of upper-triangular matrices,
and $T$ its maximal torus consisting of diagonal matrices.

For the case $(G,K) = (GL_{2n},Sp_{2n})$, we realize $K$ as the subgroup of $G$ which
preserves the antisymmetric form $\left\langle \cdot,\cdot \right\rangle$ defined by
$\left\langle e_i,e_j \right\rangle = \delta_{i,2n+1-j}$ for $i<2n+1-j$.
Thus $K$ is the subgroup of $G$ fixed by the involution
\[ \theta(g) = J (g^{-1})^t J, \]
where $J$ is the $2n \times 2n$ matrix whose antidiagonal consists of $n$ many $1$'s followed by $n$ many $-1$'s (reading
northeast to southwest), with $0$'s elsewhere.

With this particular realization of $K$, $K \cap B$ is a Borel subgroup of $K$, and $S:=K \cap T$ is a maximal torus
of $K$.  $S$ consists of diagonal matrices of the form
\[ s = \text{diag}(a_1,\hdots,a_n,a_n^{-1},\hdots,a_1^{-1}). \]

It is the action of this torus $S$ on $K$-orbit closures which we consider when discussing equivariant cohomology or
equivariant $K$-theory classes.

Let $Y_1,\hdots,Y_n$ denote the standard coordinate functions on $S$ (i.e. $Y_i(s) = a_i$, in the notation above), and
let $X_1,\hdots,X_{2n}$ be the standard coordinate functions on $T$.  The representation ring $R(S)$ (resp. $R(T)$)
is isomorphic to $\oplus_{\lambda \in X(S)} \C \cdot e^{\lambda}$ (resp. $\oplus_{\lambda \in X(T)} \C \cdot e^{\lambda}$).  The restriction map
$\rho: R(T) \rightarrow R(S)$ induced by the inclusion $S \subseteq T$ is defined by
\[ \rho(e^{X_i}) = 
 \begin{cases}
  e^{Y_i} & \text{ if $i \leq n$} \\
  e^{-Y_{2n+1-i}} & \text{ if $i > n$.}
 \end{cases}
\]

As stated in \cite[pg.~128]{Brion-99} and used in \cite{Wyser-13b} (and as is easily computed directly in our examples), in fact 
\[(G/B)^S = (G/B)^T.\] 
Also, 
given our realization of $K$, the unique closed orbit $\frX_{w_0}$ contains precisely the $S$-fixed points
corresponding to mirrored permutations $w \in S_{2n}$, i.e. permutations $w$ with the property that
\[ w(2n+1-i) = 2n+1-w(i) \text{ for $i=1,\hdots,n$}. \]

Such permutations provide the standard embedding of the hyperoctahedral group into $S_{2n}$. 
Alternatively, recall the standard bijection of these permutations with signed permutations on $\{\pm 1, \hdots, \pm n\}$:  Given a mirrored permutation $w$, define the signed permutation $\sigma_w$ first on $\{1,\hdots,n\}$ by 
\[ \sigma_w(i) = 
	\begin{cases}
		w(i) & \text{ if $w(i) \leq n$} \\
		-(2n+1-w(i)) & \text{ otherwise,}
	\end{cases}
\]
then declare that $\sigma_w(-i) = -\sigma_w(i)$. Conversely, given a signed permutation $\sigma$, it embeds as $w_{\sigma} \in S_{2n}$, defined on $\{1,\hdots,n\}$ by
\[ w_{\sigma}(i) = 
	\begin{cases}
		\sigma(i) & \text{ if $\sigma(i) > 0$} \\
		2n+1-|\sigma(i)| & \text{ otherwise,}
	\end{cases}
\]
and then on $\{n+1,\hdots,2n\}$ by $w_{\sigma}(2n+1-i) = 2n+1-w_{\sigma}(i)$.

Recall that $S_{2n}$ acts on the coordinate functions $X_i$ on $T$ via
permutation of the indices, and hence on $R(T)$.  On the other hand, signed permutations of $\{\pm 1,\hdots,\pm n\}$
act on the coordinates $Y_i$ of $S$ (via permutation of the indices together with sign changes), hence also on $R(S)$.

We make the following simple observation, which is easily checked:  If $w \in S_{2n}$ is mirrored, and if $\rho$ is the
restriction map $R(T) \rightarrow R(S)$ defined above, we have
\begin{equation}\label{eqn:w-action-commutes-with-restriction}
	\rho(w(e^{X_i})) = \sigma_w(\rho(e^{X_i})).
\end{equation}

Now, we consider the case $(G,K)=(GL_n,O_n)$.  There are many similarities to the symplectic
case.  We replace the antisymmetric form by a symmetric one, defined by
\[ \left\langle e_i,e_j \right\rangle = \delta_{i,n+1-j}. \]

With this choice of realization, if $n$ is even, then all of the notations, conventions, and definitions
above for the symplectic case apply here.  (It should be noted that in this case, the closed orbit $\caY_{w_0}$ has $2$ connected components instead of one, with one component containing half of the $S$-fixed points, and the other containing the other half.  However, as we will see, this is irrelevant to our computations.)

When $n=2m+1$ is odd, there is only a slight difference in the form of the torus $S$, the notion of a signed permutation,
and the definition of the restriction map $R(T) \rightarrow R(S)$.  Indeed, in the odd case, $S$ consists of diagonal
matrices of the form
\[ s = \text{diag}(a_1,\hdots,a_m,1,a_m^{-1},\hdots,a_1^{-1}). \]

A mirrored permutation is an element $w \in S_{2m+1}$ such that $w(2m+2-i) = 2m+2-w(i)$ for
$i=1,\hdots,m$.  (Note that this forces $w(m+1) = m+1$.)  Such permutations still correspond to signed permutations of $\{\pm 1,\hdots,\pm m\}$, in the same way as in the even case.

The restriction $\rho$ is defined by $e^{X_i} \mapsto e^{Y_i}$ for $i=1,\hdots,m$, $e^{X_{m+1}} \mapsto 1$, and
$e^{X_i} \mapsto e^{-Y_{n+1-i}}$ for $i=m+2,\hdots,n$.

\subsection{Background on equivariant $K$-theory and the localization theorem}\label{sec:background-k-theory}
We now recall some standard material. Our references are \cite{Kostant-Kumar,Chriss-Ginzburg}.

$K_0^S(X)$ denotes the Grothendieck group of $S$-equivariant coherent sheaves on an $S$-variety $X$, while $K^0_S(X)$ denotes the Grothendieck group of $S$-equivariant locally free sheaves on $X$.  When $X$ is a smooth variety, such as $G/B$, these groups are isomorphic.  Tensor product of vector bundles gives $K_S^0(X)$ a natural ring structure.  We primarily consider $K_S^0(X)$.

$K_S^0(-)$ is contravariant for $S$-equivariant maps.  Letting $X$ be any $S$-variety, the
map $X \rightarrow \{ \text{pt.} \}$ gives a pullback map $K_S^0(\{ \text{pt.} \}) \rightarrow K_S^0(X)$, giving
$K_S^0(X)$ the structure of a $K_S^0(\{ \text{pt.} \})$-module.  The ring $K_S^0(\{ \text{pt.} \})$ is isomorphic to
$R(S)$, the representation ring of $S$, mentioned in Section \ref{sec:conventions}.

In our setting, when $S \subseteq T$ are the maximal tori of $K$ and $G$, respectively, defined in Section
\ref{sec:conventions}, $K_S^0(X)$ can be described as follows:
\[ K_S^0(G/B) \cong R(S) \otimes_{R(T)^W} R(T). \]
The isomorphism (from the right-hand side to the left-hand side) is given by
\[ e^{\mu} \otimes e^{\lambda} \mapsto e^{\mu} \cdot [\mathcal{L}_{\lambda}], \]
where $\mathcal{L}_{\lambda} := G \times^B \C_{\lambda}$ is the standard line bundle on $G/B$ constructed from a line on
which $B$ acts with weight $\lambda$, and where $\cdot$ denotes the aforementioned $R(S)$-module structure arising
from pullback through the map to a point.

Identifying $R(S)$ with $\C[y_1^{\pm 1},\hdots,y_n^{\pm 1}]$ and $R(T)$ with $\C[x_1^{\pm 1},\hdots,x_{2n}^{\pm 1}]$ 
allows one to recover the description of $K_S^0(G/B)$ given in Section \ref{sec:intro-k-theory}.  In particular, via 
this identification, the $\yy$-variables are classes pulled back from $K_S^0(\{ \text{pt.} \})$, while the
$\xx$-variables are the (classes of) standard line bundles on $G/B$.

The localization theorem for equivariant $K$-theory implies the following:
\begin{Theorem}\label{thm:k-theory-localization}
For $X=G/B$, the pullback map
$K^0_S(X) \rightarrow K^0_S(X^S)$
induced by the inclusion $X^S \hookrightarrow X$ is injective.
\end{Theorem}

When $X^S$ is finite, as it is in our cases, Theorem~\ref{thm:k-theory-localization} says an 
equivariant $K$-theory class is
determined by its restrictions to the fixed points.  This will be our method to verify the correctness
of our formulas.

We also remark on the restriction maps to fixed points.  Recall that the variables $x_i$ in $K_S^0(G/B)$ represent the classes of standard torus-equivariant line bundles $\mathcal{L}_{X_i}$.  If $i_w$ denotes the inclusion of the fixed point $w=wB/B$ into $G/B$, then restriction at $w$ acts by permutation of the indices on the $\xx$-variables, i.e. $i_w^*(x_i) = e^{X_{w(i)}}$.  This is essentially because the full torus $T$ of $G$ acts on the fiber $(\caL_{X_i})_w$ with weight $w(X_i) = X_{w(i)}$.  (Note that this describes the $T$-equivariant restriction map.  Since we are working in $S$-equivariant $K$-theory, the permutation action must be followed by the restriction map $\rho$ defined in Section \ref{sec:conventions}.)

Recall also that the product structure on $K$-theory arises from tensor product of vector bundles, and that $\mathcal{L}_{\lambda} \otimes \mathcal{L}_{\mu} \cong \mathcal{L}_{\lambda + \mu}$.  Thus when restricting a product of $\xx$-variables, the associated characters add, i.e. \[i_w^*(x_ix_j) = e^{X_{w(i)} + X_{w(j)}}.\]  

\subsection{Proofs of Theorem \ref{thm:k-formula-sp} and Proposition \ref{prop:k-formula-o}}
\label{sec:k-theory-formula-proofs}
Since the map from $S$-equivariant $K$-theory to ordinary $K$-theory is to simply set all $\yy$-variables to $1$, and since the representatives described by Theorem \ref{thm:k-formula-sp} and Proposition \ref{prop:k-formula-o} do not use any $\yy$-variables, it is clear that we have only to verify that these representatives are correct $S$-equivariantly.

We adapt the arguments 
of  \cite{Wyser-13b} here to
equivariant $K$-theory, using the facts from Section \ref{sec:background-k-theory}.
We use Theorem \ref{thm:k-theory-localization}, combined with the following $K$-theoretic version of the self-intersection formula, due to R.W. Thomason \cite[Lemma 3.3]{Thomason}.  For brevity, we state only the particular consequence of 
Thomason's lemma we need:

\begin{Lemma}\label{lem:self-intersection}
Fix an algebraic torus $S$.  Let $X$ be a smooth $S$-variety, and let $j: Z \hookrightarrow X$ be an $S$-equivariant regular embedding of a smooth $S$-stable subvariety $Z$.  Let $\caN^*$ denote the conormal bundle to $Z$ in $X$, and use the shorthand 
\[ \lambda_{-1}(\caN^*) := \sum_{i=0}^{\infty} \left( -1 \right)^i \left[ \bigwedge^i \caN^* \right] \in K_0^S(Z). \]
Then in $K_0^S(Z)$, we have 
\[ j^* \circ j_*([\caO_Z]) = \lambda_{-1}(\caN^*). \]
\end{Lemma}

We describe Lemma~\ref{lem:self-intersection}'s use to 
compute the restriction of the $K$-class in question to an
$S$-fixed point.  Choose an $S$-fixed point $w:=wB/B$, and let $i$ be the inclusion of $w$ into the
closed orbit $\frX_{w_0}$.  Let $j$ denote the inclusion of $\frX_{w_0}$ into $G/B$, and let $k=j \circ i$ be the inclusion of $w$ into $G/B$.  Let $S(w)$
denote the multiset of weights of the $S$-action on $\caN:=\caN_{\frX_{w_0}}(G/B)|_w$, the normal bundle to  $\frX_{w_0}$ in $G/B$
restricted to the point $w$.

Then the restriction $[\caO_{\frX_{w_0}}]|_w$ of the class $[\caO_{\frX_{w_0}}] \in K_S^0(G/B)$ to $w$ is given by
\begin{equation}\label{eqn:compute-restriction}
 k^*(j_*([\caO_{\frX_{w_0}}])) = i^*(j^* \circ j_*([\caO_{\frX_{w_0}}])) = i^*(\lambda_{-1}(\caN^*))
 = \displaystyle\prod_{\chi \in S(w)} (1-e^{-\chi}).
\end{equation}

Note that the second of the string of equalities above uses Lemma \ref{lem:self-intersection}, taking $Z=\frX_{w_0}$ and $X=G/B$.  The last equality is evident, since the product on the right-hand side expands as an alternating sum of elementary symmetric polynomials in the weights $e^{-\chi}$.  It is clear from the definition of $\lambda_{-1}(\caN^*)$ (cf. Lemma \ref{lem:self-intersection}) that this alternating sum is what results when one restricts $\lambda_{-1}(\caN^*)$ to $w$.

Now, the weights in $S(w)$ can be computed, as $\caN|_w$ is simply the quotient of tangent spaces
$T_w(G/B)/T_w \frX_{w_0}$, considered as an $S$-module.  The tangent space $T_w(G/B)$ is well-understood, while
$T_w \frX_{w_0}$ is easily computable since $\frX_{w_0}$ is isomorphic to the flag variety for $K=Sp(2n,\C)$
\cite[Proposition 5]{Wyser-13b}.  As found in \cite[Proposition 6]{Wyser-13b}, we have
\[ S(w) = \{ \rho(w\Phi^+) \setminus \Phi_K \}, \]
where $\rho: R(T) \rightarrow R(S)$ is the restriction.  Note that this a multiset in general, since $\rho(w\Phi^+)$
can contain weights with multiplicity greater than $1$.

To verify the correctness of the formula of Theorem \ref{thm:k-formula-sp}, it remains to
check both
\begin{enumerate}
	\item[(A)] For any $S$-fixed point $w \in \frX_{w_0}$, the formula restricts at $w$ to give $\displaystyle\prod_{\chi \in S(w)} (1-e^{-\chi})$.
	\item[(B)] For any $S$-fixed point $w \notin \frX_{w_0}$, the formula restricts at $w$ to give $0$.
\end{enumerate}

As explained in Section \ref{sec:conventions}, $\frX_{w_0}$ contains precisely the $S$-fixed points
corresponding to mirrored permutations $w \in S_{2n}$.  We start by computing $S(w)$ for such a $w$.

Using \eqref{eqn:w-action-commutes-with-restriction}, we can compute the set $S(w)$ as follows:
Take the weights of $T_1(G/B)$, restrict them to $S$, and discard
(with multiplicity $1$) any which occur as a weight on $T_1 \frX_{w_0}$.  Then, apply $w$ (considered as a signed
permutation) to the resulting multiset of weights.

The weights on $T_1(G/B)$ correspond to the positive roots such that the roots of $B$ are
negative.  Since $B$ was chosen to be the upper-triangular Borel, these are $-X_i+X_j$ where $1 \leq i < j \leq 2n$.
Restricting these to $S$, we have the following weights:
\begin{itemize}
 \item $-Y_i \pm Y_j$ for $1 \leq i < j \leq n$ (each with multiplicity $2$);
 \item $-2Y_i$ for $1 \leq i \leq n$ (each with multiplicity $1$).
\end{itemize}

Discarding weights of $T_1 \frX_{w_0}$ (that is, roots of $K$) with multiplicity $1$, we are left only with 
weights of the form $-Y_i \pm Y_j$ ($1 \leq i < j \leq n$), each with multiplicity $1$.  The weights of $S(w)$
can be obtained from these weights by applying $w$, considered as a signed permutation. Using \eqref{eqn:compute-restriction}
together with \eqref{eqn:w-action-commutes-with-restriction} again, we see that 
\begin{equation}\label{eqn:explicit-restriction}
[\caO_{\frX_{w_0}}]|_w = \displaystyle\prod_{1 \leq i < j \leq n} (1-w(e^{Y_i + Y_j}))(1-w(e^{Y_i - Y_j})).
\end{equation}

We now show that $\Upsilon:=\Upsilon_{w_0,(GL_{2n},Sp_{2n})}^K$ is correct, by checking that both 
(A) and (B) hold.

Note that 
\[ \Upsilon = \displaystyle\prod_{1 \leq i < j \leq 2n-i} (1-x_ix_j) = \displaystyle\prod_{1 \leq i < j \leq n} (1-x_ix_j)(1-x_ix_{2n+1-j}). \]

Let $w \in \frX_{w_0}$ be a mirrored permutation.  Then using the latter expression for $\Upsilon$, we have
\begin{eqnarray}\nonumber
 \Upsilon|_w = \Upsilon(\rho(w \xx),\yy) & = & \prod_{1 \leq i < j \leq n} (1-\rho(e^{X_{w(i)}+X_{w(j)}}))(1-\rho(e^{X_{w(i)}+X_{w(2n+1-j)}}))\\ \nonumber
& = &\prod_{1 \leq i < j \leq n} (1-w(e^{Y_i+Y_j}))(1-w(e^{Y_i-Y_j})), \nonumber
\end{eqnarray}
using \eqref{eqn:w-action-commutes-with-restriction} once more.  As we saw above in \eqref{eqn:explicit-restriction}, this is precisely what $\Upsilon|_w$ is required to be.

Next, we show that if $w \notin \frX_{w_0}$, then
$\Upsilon|_w = 0$. 
If $w \notin \frX_{w_0}$, this means that $w$ is not a mirrored permutation.  Thus there is some smallest index $i$
such that $w(2n+1-i) \neq 2n+1-w(i)$.  Letting $j=2n+1-w(i)$, and letting $k=w^{-1}(j)$, it is clear that
$1 \leq i < k \leq 2n-i$, so that $1-x_ix_k$ divides $\Upsilon$.  Applying restriction at $w$ to this particular
factor gives
\[ 1-\rho(e^{X_{w(i)}})\rho(e^{X_{w(k)}}) = 1-\rho(e^{X_{w(i)}})\rho(e^{X_{2n+1-w(i)}}) = 1-e^{Y_l}e^{-Y_l} = 0, \]
for some $l$.  Thus $\Upsilon|_w = 0$, as required.

We conclude that $\Upsilon$ represents $[\caO_{\caY_{w_0}}]$.

Note that the above argument, with a very minor modification, also applies to prove the correctness of the formula of
Proposition \ref{prop:k-formula-o}.  Indeed, the only difference is that when discarding roots of $K$ from the 
(restricted) weights of $T_1(G/B)$, we no longer discard those of the form $-2Y_i$, since these are not roots in 
types $B$ or $D$, whereas they are in type $C$.  Thus in either type, one computes the restriction
$[\caO_{\caY_{w_0}}]|_w$ as follows:
\[ [\caO_{\caY_{w_0}}]|_w = \displaystyle\prod_{1 \leq i \leq \lfloor n/2 \rfloor} (1-w(e^{-2Y_i})) \displaystyle\prod_{1 \leq i < j \leq n} (1-w(e^{Y_i + Y_j}))(1-w(e^{Y_i - Y_j})). \]

The argument proceeds from there unchanged, with the additional factors of $1-x_i^2$ ($i=1,\hdots,\lfloor n/2 \rfloor$) present
in the polynomial $\Upsilon_{w_0; (GL_n,O_n)}^K$ providing the additional needed factors upon restriction.

This proves Proposition \ref{prop:k-formula-o}.\qed

\begin{Remark}\label{rmk:k-to-c}
In the conventions we are using (which match those of \cite{Knutson.Miller:annals}), cohomological formulas for a class can be derived from $K$-theoretic formulas for the same class by making the
substitution $\xx \mapsto {\bf 1 - x}, \yy \mapsto {\bf 1-y}$ and then taking the sum of the lowest degree terms (cf. \cite[Remark 2.3.5]{Knutson.Miller:annals}).  Note that the formula for $\Upsilon_{w_0; (GL_n,O_n)}$ given in \eqref{eqn:GOeven} is related
to the $K$-theoretic formula of Proposition \ref{prop:k-formula-o} precisely this way.  Thus 
our proof of the latter also proves the former.
\qed
\end{Remark}

When $n$ is even, the closed $O_n$-orbit on $GL_n/B$ has two components, each being a distinct closed $SO_n$-orbit.
In this case, it would be preferable to have a formula for the $S$-equivariant $K$-class of each component individually,
rather than simply a formula for the $K$-class of their union.  In equivariant cohomology, this is done in 
\cite[Proposition 10]{Wyser-13b}, but for $K$-theory, we have been unable to find a general formula.
\begin{Problem}
For even $n$, give explicit formulas for the ($S$-equivariant) $K$-theory classes of the (two) connected components of $\caY_{w_0}$.
\end{Problem}

\subsection{Proof of Theorem \ref{thm:well-defined-stable-k-theory}}
We now prove Theorem \ref{thm:well-defined-stable-k-theory} by indicating how to modify the proofs of Theorems
\ref{thm:self-consistent} and \ref{thm:stability} to apply to $K$-theory.

To prove that applying Demazure operators gives a well-defined family $\{\Upsilon_{\pi; (GL_{2n},Sp_{2n})}^K\}$
of polynomials, we simply replace Schubert polynomials $\Schub_w$ by Grothendieck polynomials $\Groth_w({\bf x})$ in the proof
of Theorem  \ref{thm:self-consistent}.
The remainder of the proof is \emph{mutatis mutandis}, although we wish to make a remark about conventions.
We are using the same conventions for (single) Grothendieck
polynomials as, e.g., \cite{Knutson.Miller:annals}. If we make the change of variables $x_i\mapsto 1-x_i$ we obtain
Grothendieck polynomials $\Groth_w({\bf 1-x})$ whose lead term is the Schubert polynomial ${\mathfrak S}_w({\bf x})$. Note that after this change of variables
\[\Upsilon_{w_0, (GL_{2n},Sp_{2n})}^K = \displaystyle\prod_{1 \leq i < j \leq 2n-i} (1 - x_ix_j)\mapsto 
 \displaystyle\prod_{1 \leq i < j \leq 2n-i} (x_i+x_j-x_i x_j).\]
Both the middle and latter expressions live in ${\caL}_{2n} \subseteq {\mathbb Q}[x_1,\ldots,x_{2n}]$.
In particular, the latter can expressed
as a linear combination of the polynomials $\Groth_w({\bf 1-x})$ for $w\in S_{2n}$. Now we obtain an expression for 
$\Upsilon_{w_0, (GL_{2n},Sp_{2n})}^K$ in terms of the $\Groth_{w}({\bf x})$ by changing variables back.

The proof of the stability of the family $\{\Upsilon_{\pi; (GL_{2n},Sp_{2n})}^K\}$ is almost identical to the proof
of Theorem \ref{thm:stability}.  One shows by induction that \[(D_{2N-2} \circ D_{2N-3} \circ \hdots
\circ D_1)(\Upsilon_{w_0; (GL_{2N},Sp_{2N})}^K) = \Upsilon_{w_0; (GL_{2(N-1)},Sp_{2(N-1)})}^K.\] 
 The first $N-1$
Demazure operators applied strip off the factors $1-x_ix_j$ with $i+j=2N$ in the order $1-x_1x_{2N-1},\hdots,
1-x_{N-1}x_{N+1}$.  The next $N-1$ Demazure operators strip off the factors $1-x_ix_j$ with $i+j=2N-1$ in the order
$1-x_{N-1}x_{N},1-x_{N-2}x_{N+1},\hdots,1-x_1x_{2N-2}$.  This can easily be proved by induction, exactly as in the proof
of Theorem \ref{thm:stability}, using the fact that $D_i(f) = -\partial_i(x_{i+1}f)$.  We omit the details.
\qed

We remark that in ordinary $K$-theory \cite[Theorem~1]{Brion-02} implies that $\Upsilon_{\pi; (GL_{2n},Sp_{2n})}^K$
expands as an alternating sum of Grothendieck polynomials. (Here were have used that the orbit closures
in this case have rational singularities.)  Since the monomials of the Grothendieck polynomials ${\mathfrak G}_w({\bf 1-x})$ 
also alternate in sign
by degree, the above argument allows one to conclude an alternation-in-sign of the monomial expansion 
of the polynomial $\Upsilon_{\pi; (GL_{2n},Sp_{2n})}^K({\bf 1- x})$, which is the $K$-theory representative, up to a 
change of convention.

\section{Final remarks}\label{sec:final-remarks}
\label{sec:expansion-double-schub}
We refer the reader to \cite[Section~2.3]{Manivel} for definitions of double Schubert polynomials ${\mathfrak S}_w$.
In brief, it is standard that any polynomial $f\in{\mathbb Z}[x_1,x_2,\ldots; y_1,y_2,\ldots]$ 
can be expressed as a ${\mathbb Z}[y_1,y_2,\ldots]$-linear combination of ${\mathfrak S}_w$ for $w\in S_{\infty}$.
Now consider the following expansion
\begin{eqnarray}\label{eqn:o3-schubert}
\Upsilon_{w_0; (GL_3,O_3)}& = & 2 x_1 (x_1 + x_2)\\ \nonumber
& = & (2y_1^2+2y_1 y_2) {\mathfrak S}_{123}+2(y_1+y_2){\mathfrak S}_{213}+2y_1{\mathfrak S}_{132}+2{\mathfrak S}_{231}+2{\mathfrak S}_{312}.\nonumber
\end{eqnarray}
This is an example of the following:
\begin{Corollary}[of Theorem~\ref{thm:self-consistent}]
\label{cor:double}
The $\Upsilon$ polynomials are a (unique) linear combination of double Schubert polynomials ${\mathfrak S}_{w}(\xx;\yy)$ with $w\in S_{n}$. The coefficients are in ${\mathbb Z}_{\geq 0}[y_1,\ldots,y_n]$.
\end{Corollary}
\begin{proof}
In fact, any (single) Schubert polynomial ${\mathfrak S}_w(\xx)$ is a linear
combination of double Schubert polynomials ${\mathfrak S}_{v}(\xx;\yy)$ of the desired sort. More precisely, we have
\[{\mathfrak S}_{w}(\xx)=\sum_{u\cdot v=w}{\mathfrak S}_{u}(\yy) {\mathfrak S}_{v}(\xx;\yy),\]
where $\ell(u)+\ell(v)=w$. This identity follows from a formula of A.~N.~Kirillov; for a proof see \cite[Corollary~1]{BKTY}.
Finally, the corollary itself holds since each $\Upsilon$-polynomial is a nonnegative sum of single Schubert
polynomials, by Theorem~\ref{thm:self-consistent}.
\end{proof}
Thus, any formula for the expansion for the $\Upsilon$-polynomials
in terms of single Schubert polynomials implies an expansion
formula in terms of double Schubert polynomials. 

M.~Brion \cite[Theorem 1.5]{Brion-98} expresses an
ordinary (non-equivariant) cohomology class of any $K$-orbit closure as an explicit weighted sum of Schubert classes.  
The formula
is in terms of a sum over paths in the weak order graph, each weighted by a certain power of $2$.  
So in principle, 
polynomial representatives for the ordinary cohomology classes of the closed orbits in our two cases
were already known, since the weak order graphs are well-understood in these cases.  One simply replaces the Schubert classes from Brion's formula by the corresponding Schubert polynomials, weighted by the corresponding multiplicities.

Our proof of Theorem \ref{thm:self-consistent} makes clear that the polynomials obtained in this way are in fact equal to our $\Upsilon$-polynomials.  Even so, the form in which we give the representatives for $\Upsilon_{w_0; (G,K)}$ (for either pair $(G,K)$) is not immediate from Brion's result, since it is not obvious that the sum of Schubert polynomials in question factors in the form
in which we present it.

Note that Brion's formula does not apply equivariantly, so our equivariant representatives cannot be deduced from it.
This is evident, for example, in the multiplicity-free case, where Brion exhibits a flat degeneration of a
multiplicity-free orbit closure to a union of Schubert varieties which is visibly not equivariant for the
torus action.  We can make this more explicit using the double Schubert expansion formula appearing in the proof of
Corollary \ref{cor:double}.  For example, for the closed $Sp_4$-orbit $\frX_{4321}$ on $GL_4/B$, we see that
\begin{eqnarray}\nonumber
\Upsilon_{w_0; (GL_4,Sp_4)}& = & (x_1+x_2)(x_1+x_3)\\ \nonumber
& = & (y_1^2+y_1y_2+y_1y_3+y_2y_3) {\mathfrak S}_{1234}(\xx;\yy)+(y_1+y_2){\mathfrak S}_{2134}(\xx;\yy)+ \\ \nonumber
& &
(y_1+y_2){\mathfrak S}_{1243}(\xx;\yy)+{\mathfrak S}_{1342}(\xx;\yy)+{\mathfrak S}_{3124}(\xx;\yy).\nonumber
\end{eqnarray}

So the $T$-equivariant class represented by $\Upsilon_{w_0; (GL_4,Sp_4)}$ is given in the Schubert basis by
\[ (y_1^2+y_1y_2+y_1y_3+y_2y_3) [X^{1234}]_T + (y_1+y_2) [X^{2134}]_T + (y_1+y_2) [X^{1243}]_T + [X^{1342}]_T + [X^{3124}]_T. \]

Restricting to $S$-equivariant cohomology (which corresponds to setting $y_3 = -y_2$ and $y_4 = -y_1$), we see that
\[
 [\frX_{4321}]_S = (y_1^2-y_2^2)[X^{1234}]+ (y_1+y_2)[X^{2134}]_S + (y_1+y_2)[X^{1243}]_S + [X^{1342}]_S + [X^{3124}]_S. \]

Similarly, using \eqref{eqn:o3-schubert} above, we see that the closed orbit $O_3$-orbit $\caY_{321}$ on $GL_3/B$ has
$S$-equivariant class
\[
 [\caY_{321}]_S = 2y_1^2[X^{123}]_S + 2y_1[X^{213}]_S + 2y_1[X^{132}]_S + 2[X^{231}]_S + 2[X^{312}]_S.\]

\section*{Acknowledgements}
We thank Bill Graham for informing us of the reference \cite{Thomason}.  We also thank Michel Brion for many helpful conversations about the technicalities of equivariant $K$-theory and Demazure operators.
AY was supported by NSF grant DMS 1201595 as well as the Helen Corley Petit
endowment at UIUC.  BW was supported by NSF International Research Fellowship 1159045 and
hosted by Institut Fourier in Grenoble.

\end{document}